\documentclass{amsart}

\usepackage[utf8]{inputenc}
\usepackage[english]{babel}

\usepackage{fullpage}

\usepackage{amsmath}
\usepackage{amsfonts}
\usepackage{amsthm}
\usepackage{amssymb}
\usepackage{theoremref}
\usepackage{colonequals}
\usepackage[super]{nth}
\usepackage{url}
\usepackage{graphicx}

\newcounter{master}
\numberwithin{master}{section}

\theoremstyle{plain}

\newtheorem{theorem}[master]{Theorem}
\newtheorem{question}[master]{Question}

\newtheorem{lemma}[master]{Lemma}

\theoremstyle{definition}

\newtheorem{definition}[master]{Definition}
\newtheorem{claim}[master]{Claim}
\newtheorem*{notation}{Notation}

\theoremstyle{remark}

\newtheorem{remark}[master]{Remark}

\makeatletter
\let\c@equation\c@master

\let\c@figure\c@master

\let\c@table\c@master

\makeatother

\makeatletter
\def\@cite#1#2{\textup{[\textbf{#1}\if@tempswa , #2\fi]}}
\def\@biblabel#1{[\textbf{#1}]}
\makeatother

\def\nth#1{
  \expandafter\nthM #1\relax
  \nthscript{%
  \ifnum0<0#1 
  \ifnum#1\nthtest0 th\else 
  \expandafter \nthSuff \expandafter 0\number\ifnum #1<0-\fi#1\delimiter
  \fi
  \else th\fi 
 }}

\def\nthM#1{\ifmmode#1\else$#1$\fi}

\DeclareMathOperator{\aff}{af{}f}
\DeclareMathOperator{\codim}{codim}

\begin{document}

\title{On star-convex bodies with rotationally invariant sections}

\author[B. Zawalski]{Bartłomiej Zawalski}
\address{Polish Academy of Sciences, Institute of Mathematics}
\email{b.zawalski@impan.pl}
\thanks{This work was supported by the Polish National Science Center grants no. 2015/18/A/ST1/00553 and 2020/02/Y/ST1/00072.}

\begin{abstract}
We will prove that an origin-symmetric star-convex body $K$ with sufficiently smooth boundary and such that every hyperplane section of $K$ passing through the origin is a~body of affine revolution, is itself a~body of affine revolution. This will give a~positive answer to the recent question asked by G.~Bor, L.~Hern\'andez-Lamoneda, V.~Jim\'enez de Santiago, and L.~Montejano-Peimbert, though with slightly different prerequisites.
\end{abstract}

\maketitle

\paragraph{\bf Acknowledgement}

This version of the article has been accepted for publication, after peer review, but is not the Version of Record and does not reflect post-acceptance improvements, or any corrections. The Version of Record is available online at: \url{http://dx.doi.org/10.1007/s13366-023-00702-1}.

\tableofcontents

\section{Introduction}

After more than five decades since the seminal works of H.~Auerbach, S.~Mazur and S.~Ulam \cite{Auerbach1935}, A.~Dvoretzky \cite{10.2307/90172}, M.~Gromov \cite{gromov1967geometrical} and V.~Milman \cite{milman1971new}, the isometric conjecture of S.~Banach again attracted the attention of researchers, launching a~whole avalanche of papers by L.~Montejano et al. \cite{10.2140/gt.2021.25.2621,article2,article1,arxiv.2108.08917} and recently also by S.~Ivanov, D.~Mamaev and A.~Nordskova \cite{Ivanov2023}. As it was already known that algebraic topology alone would not suffice, more sophisticated methods were developed. For instance, in \cite{10.2140/gt.2021.25.2621} the authors (G.~Bor, L.~Hern\'andez-Lamoneda, V.~Jim\'enez de Santiago and L.~Montejano-Peimbert) showed that under assumptions of the conjecture (namely, that $K$ is a~symmetric convex body, all of whose hyperplanar sections are affinely equivalent) supplemented with dimension constraints having its origins in algebraic topology, all the hyperplanar sections of $K$ must be bodies of affine revolution (cf. \thref{def:02}). This observation prompted them to ask the following, somewhat more general question:

\begin{question}[{cf. \cite[Remark~2.9]{10.2140/gt.2021.25.2621}}]\thlabel{qn:01}
Let $K\subset\mathbb R^n$, $n\geq 4$, be a~convex body containing the origin $O$ in its interior. If every hyperplane section of $K$ passing through $O$ is a~body of affine revolution, is $K$ necessarily a~body of affine revolution?
\end{question}

\noindent Note that the reverse implication is quite straightforward (cf. \cite[Lemma 2.4]{10.2140/gt.2021.25.2621}). Moreover, the authors proved in \cite[Theorem 1.4]{10.2140/gt.2021.25.2621} that at least one hyperplane section of such a~symmetric convex body must be an ellipsoid, which is an obvious necessary condition. Compared to the initial problem of S. Banach, they decided to keep the assumption that $K$ is convex while forgoing the assumption that $K$ is symmetric. In what follows, we will prove a~theorem in the same spirit, but with slightly different prerequisites:

\begin{theorem}\thlabel{thm:01}
Let $K\subset\mathbb R^n$, $n\geq 4$, be an origin-symmetric star-convex body. Assume that the boundary $\partial K$ is a~submanifold of class $C^3$. If every hyperplane section of $K$ passing through the origin is a~body of affine revolution, then $K$ itself is a~body of affine revolution.
\end{theorem}

Our argument is rather elementary. It is built mainly upon the tools of differential geometry and linear algebra. Although occasionally we will need to use some more involved facts from other fields like algebraic topology or commutative algebra, they will hide most of the difficulty within themselves. Unlike in \cite{10.2140/gt.2021.25.2621}, we forgo the assumption that $K$ is convex while keeping the assumption that $K$ is symmetric. Moreover, to apply our method we need the boundary of $K$ to be sufficiently smooth. Presumably, the superfluous symmetry assumption can be disposed of, but this will significantly complicate any proof along our lines and most likely it will also lose its nice geometric flavor to the intensive computation of general affine differential invariants (cf. \thref{rem:01}). The smoothness assumption seems to be an inherent element of our argument and therefore can not be easily relaxed.\\

A natural question arises if the assumption $n\geq 4$ is indeed necessary. A compact domain $L\subseteq\mathbb R^{n-1}$ is a~body of affine revolution if its symmetry group contains a~subgroup affinely conjugated to $\mathrm O(n-2,\mathbb R)$ (cf. \thref{def:02}). In dimension $n=3$, \thref{qn:01} has a~different flavor because we assume merely that every planar section of $K$ passing through the origin admits an affine reflection, which is satisfied e.g. when $K$ is a~cube (every central planar section of a~cube is affinely equivalent to either a~square or a~regular hexagon, both of which are axially symmetric). Therefore the statement is no longer true unless we make some additional assumptions (see e.g. \cite[\S 2]{article3}). The right counterpart of \thref{qn:01} in dimension $3$ seems to be an affine version of a~similar question asked by K.~Bezdek:

\begin{question}[{cf. \cite[\S1.4]{Odor1999}}]\thlabel{qn:02}
Let $K\subset\mathbb R^3$ be a~convex body. If every planar section of $K$ [not necessarily passing through the origin ed.] admits an affine reflection, is $K$ necessarily a~body of affine revolution?
\end{question}

\noindent T.~\'Odor claimed to have confirmed Bezdek's conjecture, but unfortunately, his approach was found incomplete. To the author's best knowledge, the problem remains open. Nevertheless, techniques similar to those presented in this paper may be applied also to \thref{qn:02}, but then they will most likely require higher-order smoothness of the boundary.

\section{Definitions and basic concepts}

We adopt the notation from \cite{10.2140/gt.2021.25.2621}.

\begin{definition}
A compact domain $K\subset\mathbb R^n$, $n\geq 1$, is called \emph{star-convex} if there exists $O\in K$ such that for every $x\in K$ the entire line segment from $O$ to $x$ is contained in $K$. A star-convex body is called \emph{symmetric} if it is centrally symmetric with respect to $O$.
\end{definition}

\begin{remark}
Actually, the same proof of \thref{thm:01} with minor technical improvements works for general compact domains. However, we will intentionally refrain from these topological considerations, so as not to overshadow the main idea.
\end{remark}

\begin{definition}\thlabel{def:02}
A compact domain $K\subset\mathbb R^n$, $n\geq 1$, is called \emph{a body of affine $k$-revolution} if its symmetry group contains a~subgroup $G$ affinely conjugated to $\mathrm O(n-k,\mathbb R)$, $0<k<n$. The ambient space $\mathbb R^n$ can be viewed as a~direct sum $H\oplus L$ of a~linear space $H$ and an affine space $L$, where $H$ (called \emph{the hyperplane of affine revolution}) is an irreducible representation space of $G$ of dimension $n-k$ and $L$ (called \emph{the hyperaxis of affine revolution}) is a~common fixed point subspace of $G$ of dimension $k$. By body (resp. hyperplane, axis) of affine revolution, we will mean a~body (resp. hyperplane, axis) of affine $1$-revolution unless expressly stated otherwise.
\end{definition}

\begin{remark}
If we additionally assume that $K$ is symmetric, then the center of symmetry $O$ must be a~fixed point of any affine symmetry of $K$. In particular, if $K$ is a~star-convex body of affine revolution, the axis of affine revolution must pass through $O$. Moreover, since every section of $K$ with a~hyperplane passing through $O$ is again a~star-convex body of affine revolution symmetric with respect to $O$, the axis of affine revolution of all such hyperplanar sections must likewise pass through $O$.
\end{remark}

\begin{remark}
Note that all these objects are defined (and will be used) in a~general affine setting. In particular, the symmetry group of $K$ is a~compact subgroup of $\mathrm{GL}(n,\mathbb R)$, but not necessarily of $\mathrm O(n,\mathbb R)$.
\end{remark}

Denote the submanifold $\partial K$ by $M^{n-1}$. Let $p\in M^{n-1}$ be any point with positive definite second fundamental form of $M^{n-1}$. After applying a~suitable affine map we may assume that $p=\boldsymbol 0_{\mathbb R^n}$, $O=\boldsymbol 0_{\mathbb R^n}+\hat{\boldsymbol e}_n$ and $T_pM=\boldsymbol 0_{\mathbb R^n}+\langle\hat{\boldsymbol e}_n\rangle^\perp$, where $\hat{\boldsymbol e}_{n}$ stands for the \nth{n} standard unit vector (fig. \ref{fig:01}). In this coordinate system, we represent the neighborhood of $p$ in $M^{n-1}$ as a~graph of some function $f:T_pM^{n-1}\supset U\to\mathbb R$ of class $C^3$, which must be of the form
$$f(\boldsymbol x)=O(\|\boldsymbol x\|)^2$$
in Big-O notation. Since we assumed that the second fundamental form of $M^{n-1}$ is positive definite at $p$, after applying a~suitable linear change of coordinates in the domain we may further assume that
$$f(\boldsymbol x)=\frac{1}{2}\langle\boldsymbol x,\boldsymbol x\rangle+O(\|\boldsymbol x\|)^3.$$
The above will be called the \emph{canonical parametrization} of $M^{n-1}$ at $p$. Note that it is unique up to an orthogonal change of coordinates in the domain. Moreover, observe that the restriction of $f$ to any codimension $1$ hyperplane $H\in\mathrm{Gr}(n-2,T_pM^{n-1})$ is the canonical parametrization of the hyperplanar section $M^{n-1}\cap\aff(\{H,O\})$ at $p$.

\begin{figure}[h]
\includegraphics{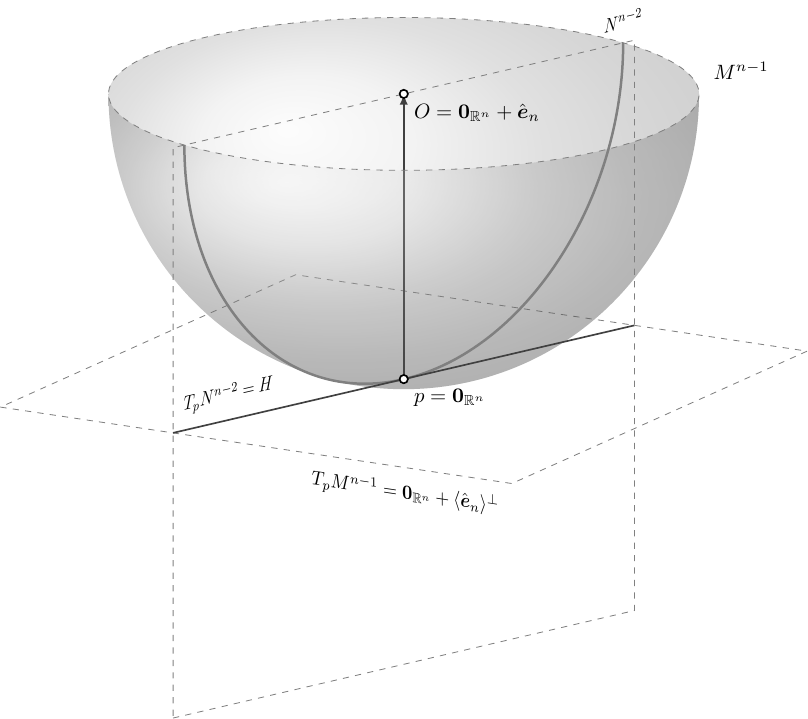}
\caption{The canonical parametrization of $M^{n-1}$ at $p$}
\label{fig:01}
\end{figure}

\begin{definition}
Let $f:\mathbb R^m\to\mathbb R$ be a~function of class $C^3$. The homogeneous part of degree $3$ of its series expansion is called the \emph{cubic form} of $f$ and will be denoted by $c_f$.
\end{definition}

\begin{remark}
In the course of the proof, we will consider almost exclusively points with positive definite second fundamental form of $M^{n-1}$. However, we do not need to assume that $M^{n-1}$ is strongly convex. Indeed, every compact hypersurface contains at least one such point, from which all the local properties will eventually spill over the entire set.
\end{remark}

\section{Hypersurfaces of affine revolution}\label{sec:01}

Although the original hypersurface is $(n-1)$-dimensional, most of the time we will be investigating $(n-2)$-dimensional hypersurfaces of affine revolution since their geometry plays a~key role in the proof. Let
$$g:T_pN^{n-2}\supset U\to\mathbb R,\quad g(\boldsymbol x)=\frac{1}{2}\langle\boldsymbol x,\boldsymbol x\rangle+O(\|\boldsymbol x\|)^3$$
be the canonical parametrization of some hyperplanar section $N^{n-2}$ of $M^{n-1}$ at $p$ (fig. \ref{fig:01}), being a~hypersurface of affine revolution. From now henceforth, by action of a~linear group we always mean the action of its affine matrix representation on a~specified affine subspace of $\mathbb R^n$, usually clear from the context. By definition, $N^{n-2}$ is invariant under action of $\mathrm O(n-2,\mathbb R)$. Denote by $G_p$ the \emph{isotropy group} of $p$, i.e. the set of affine symmetries of $N^{n-2}$ which does not change $p$. If $p$ is already a~fixed point of $\mathrm O(n-2,\mathbb R)$ then $G_p$ is affinely conjugated to $\mathrm O(n-2,\mathbb R)$, otherwise $G_p$ is affinely conjugated to $\mathrm O(n-3,\mathbb R)$. Without loss of generality, we may choose $U$ to be invariant under $G_p$.\\

\begin{figure}[h]
\includegraphics{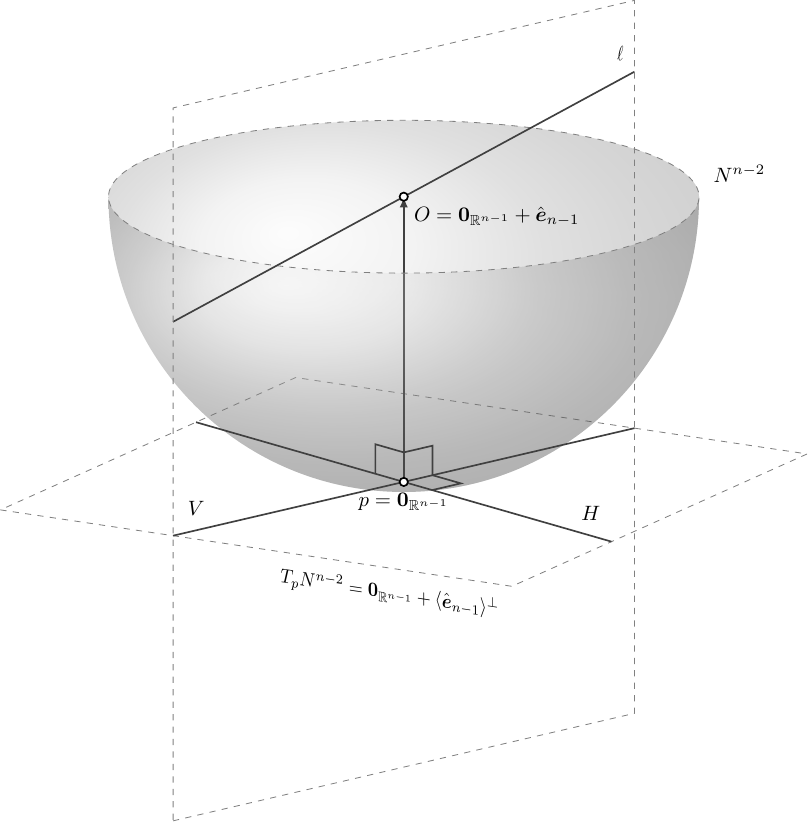}
\caption{The canonical parametrization of $N^{n-2}$ at $p$}
\label{fig:02}
\end{figure}

Let $A\in G_p$ be any affine symmetry of $N^{n-2}$ which does not change $p$. Note that in our coordinate system, $A$ may be regarded as a~linear map. Since $O$ is the center of symmetry of $N^{n-2}$, it must be a~fixed point of $A$. Thus $\hat{\boldsymbol e}_{n-1}$ is an eigenvector of $A$ with eigenvalue $+1$. Moreover, the hyperplane $\langle\hat{\boldsymbol e}_{n-1}\rangle^\perp$ it tangent to $N^{n-2}$ at $p$ and thus it must be an invariant subspace of $A$. It follows that the matrix representation of $A$ in our canonical coordinate system is of the form
$$[A]=\begin{pmatrix} [B] & \begin{matrix}0\\0\\\vdots\\0\end{matrix} \\ \begin{matrix}0&0&\cdots&0\end{matrix} & 1 \end{pmatrix}$$
for some $B\in\mathrm{GL}(n-2,\mathbb R)$. Now, for every point $\boldsymbol x\in U$ there exists a~point $\tilde{\boldsymbol x}\in U$ such that
$$[A]\mathbin{.}\begin{pmatrix}\boldsymbol x\\g(\boldsymbol x)\end{pmatrix}=\begin{pmatrix}\tilde{\boldsymbol x}\\g(\tilde{\boldsymbol x})\end{pmatrix},$$
which reads
\begin{equation}\label{eq:04}g(B\boldsymbol x)=g(\boldsymbol x).\end{equation}
In particular, $[B]$ must preserve the standard quadratic form, in which case it is an orthogonal matrix. Thus $A$ itself must be an orthogonal map, which means that in our coordinate system, $G_p$ is actually a~subgroup of $\mathrm O(n-2,\mathbb R)$.

\begin{claim}\thlabel{lem:05}
It follows immediately from \eqref{eq:04} that the canonical parametrization $g$ is invariant under action of $G_p$ on its domain $U\subset T_pN^{n-2}$, i.e. $g\circ A\vert_U\equiv g$ for every $A\in G_p$.\hfill$\square$
\end{claim}

\begin{claim}\thlabel{lem:03}
The tangent space $T_pN^{n-2}$ may be viewed as a~$(n-2)$-dimensional representation space of $G_p$. If $G_p\simeq\mathrm O(n-3,\mathbb R)$, then $T_pN^{n-2}$ admits an orthogonal decomposition $H\oplus V$ into irreducible representations' spaces, where $H$ is a~codimension $1$ hyperplane of revolution and $V$ is a~dimension $1$ common fixed-point subspace (fig. \ref{fig:02}). In particular, the cubic form $c_g$ vanishes on $H$. Indeed, $c_g\vert_H$ must vanish at some direction and by \thref{lem:05} this carries over to all the other directions as well. On the other hand, if $G_p\simeq\mathrm O(n-2,\mathbb R)$, then $T_pN^{n-2}$ is already an irreducible representation's space of $G_p$ and thus $c_g$ vanishes identically (again, by the very same argument). The latter is necessarily the case when $N^{n-2}$ is an ellipsoid.\hfill$\square$
\end{claim}

Let us recall a~simple fact from the original paper \cite{10.2140/gt.2021.25.2621}:

\begin{lemma}[{\cite[Lemma 2.3]{10.2140/gt.2021.25.2621}}]\thlabel{lem:02}
A symmetric body of affine revolution $K\subset\mathbb R^m$, $m\geq 3$, admitting two different hyperplanes of affine revolution, is an ellipsoid.
\end{lemma}

Now we are ready to prove the following key lemma, which will eventually enable us to figure out the geometry of $M^{n-1}$:

\begin{lemma}\thlabel{lem:01}
In the above setting, there exists a~codimension $1$ hyperplane $H\in\mathrm{Gr}(n-2,T_pM^{n-1})$ such that the cubic form $c_f\vert_H$ is identically zero (i.e. $c_f$ is reducible).
\end{lemma}

Interestingly enough, the proof for $n=4$ and $n\geq 5$ will be essentially different. In the first case, we need an argument from general topology, which holds only in even dimensions $n$. In the second case, we introduce an argument from algebraic geometry, which holds only in dimensions $n\geq 5$.

\begin{proof}[Proof of \thref{lem:01} for $n\geq 5$]
Suppose that $c_f$ is irreducible. Theorem of Bertini \cite[Theorem 17.16]{harris1992algebraic} asserts that there exists a~codimension $1$ hyperplane $H\in\mathrm{Gr}(n-2,T_pM^{n-1})$ such that $c_f\vert_H$ is again irreducible. However, it follows from \thref{lem:03} that $c_g=c_f\vert_H$ vanishes on some codimension $1$ hyperplane, i.e. admits a~factor of degree $1$, a~contradiction.
\end{proof}

\begin{proof}[Proof of \thref{lem:01} for $n=4$]
If there exists a~hyperplanar section of $M^3$ passing through $p$ that admits more than one axis of affine revolution, then by \thref{lem:02} and \thref{lem:03} we are done. Further, if there exists a~hyperplanar section of $M^3$ passing through $p$ such that its axis of affine revolution also passes through $p$, then again by \thref{lem:03} we are done. Therefore we may assume that every hyperplanar section of $M^3$ passing through $p$ admits exactly one axis of affine revolution, which does not pass through $p$.\\

In this case, we can define the following distribution on $\mathrm{Gr}(2,T_pM^3)$: for every plane $\pi\in\mathrm{Gr}(2,T_pM^3)$, let $\ell_\pi\subset\pi=T_\pi\mathrm{Gr}(2,T_pM^3)$ be the orthogonal projection of the (unique) axis of affine revolution of $M^3\cap\aff(\{\pi,O\})$ on $\pi$, which we already know is always a~$1$-dimensional linear subspace of $\pi$. Moreover, the map $\pi\mapsto\ell_\pi$ is clearly continuous (cf. \cite[Lemma 2.8]{10.2140/gt.2021.25.2621}), which gives rise to a~rank-$1$ subbundle $\eta$ of $T\mathrm{Gr}(2,T_pM^3)$. Now, its Stiefel-Whitney class $w_1(\eta)\in H^1(\mathrm{Gr}(2,T_pM^3);\mathbb Z/2\mathbb Z)=\{0\}$ must be $0$ and thus $\eta$ is orientable \cite[Problem~12-A]{milnor1974characteristic}. Selecting for each fiber of $\eta$ the positively oriented unit vector gives rise to a~non-vanishing vector field on $\mathrm{Gr}(2,T_pM^3)$, a~contradiction.
\end{proof}

Since the canonical parametrization is defined up to an orthogonal change of coordinates in the domain, without loss of generality we may further assume that $c_f$ vanishes on the hyperplane $\langle\hat{\boldsymbol e}_{n-1}\rangle^\perp$, i.e.
\begin{equation}\label{eq:01}c_f(\boldsymbol x)=x_{n-1}\cdot q_f(\boldsymbol x),\end{equation}
where $q_f$ is some quadratic form, not necessarily non-zero.

\begin{claim}\thlabel{lem:04}
For every irreducible quadric $Q^{n-2}\subset T_pM^{n-1}$, there exists an open subset $V\subseteq\mathrm{Gr}(n-2,T_pM^{n-1})$ of hyperplanes $H$ such that $Q^{n-2}\cap H$ contains no codimension $1$ linear subspace. Indeed, every linear space contained in an irreducible quadric has dimension at most half the dimension of the quadric \cite[Theorem 22.13]{harris1992algebraic}. Therefore if $n\geq 5$, the conclusion is trivial. For $n=4$, every irreducible quadric is projectively equivalent to either a~cone, a~straight line, or a~single point. In each case, there exists an open subset of planes that intersect $Q^2$ only at the origin.\hfill$\square$
\end{claim}

Now, if the quadratic form $q_f$ on the right-hand side of \eqref{eq:01} is irreducible, then from \thref{lem:04} it follows that for every $H\in V$ the zero set of $c_f\vert_H$ contains exactly one codimension $1$ hyperplane, namely $H\cap\langle\hat{\boldsymbol e}_{n-1}\rangle^\perp$. In particular, by \thref{lem:03}, $M^{n-1}\cap\aff(\{H,O\})$ is invariant under action of $\mathrm O(n-3,\mathbb R)$ with hyperplane of revolution $H\cap\langle\hat{\boldsymbol e}_{n-1}\rangle^\perp$. On the other hand, if the quadratic form $q_f$ is reducible, then $c_f$ can be decomposed into a~product of three linear forms, and hence its zero set is a~sum of three (not necessarily different) hyperplanes $H_1,H_2,H_3$. The same argument shows that for every $H\in\mathrm{Gr}(n-2,T_pM^{n-1})\setminus\{H_1,H_2,H_3\}$, $M^{n-1}\cap\aff(\{H,O\})$ is invariant under action of $\mathrm O(n-3,\mathbb R)$ with hyperplane of revolution $H\cap H_i$ for some $i\in\{1,2,3\}$. Denote by $V_i$ the set of hyperplanes $H\in\mathrm{Gr}(n-2,T_pM^{n-1})$ such that $M^{n-1}\cap\aff(\{H,O\})$ is invariant under action of $\mathrm O(n-3,\mathbb R)$ with hyperplane of revolution $H\cap H_i$, $i=1,2,3$. Since each $V_i$ is closed (cf. \cite[Lemma 2.7]{10.2140/gt.2021.25.2621}) and $V_1\cup V_2\cup V_3=\mathrm{Gr}(n-2,T_pM^{n-1})$, at least one of those sets has non-empty interior. After a~suitable change of coordinates, we may assume that this is the set corresponding to the plane $\langle\hat{\boldsymbol e}_{n-1}\rangle^\perp$.

\begin{claim}\thlabel{lem:08}
In either case, we are eventually in a~position where we have an open subset $V\subseteq\mathrm{Gr}(n-2,T_pM^{n-1})$ such that for every $H\in V$, $M^{n-1}\cap\aff(\{H,O\})$ is invariant under the action of $\mathrm O(n-3,\mathbb R)$ with hyperplane of revolution $H\cap\langle\hat{\boldsymbol e}_{n-1}\rangle^\perp$.\hfill$\square$
\end{claim}

\begin{notation}
For any $2$-dimensional plane $\pi\in\mathrm{Gr}(2,T_pM^{n-1})$ and any point $\boldsymbol a\in T_pM^{n-1}$, denote by $\mathrm{Ref}_\pi(\boldsymbol a)$ the orthogonal reflection of $\boldsymbol a$ across the plane $\pi$. Further, for any angle $\alpha\in\mathbb R$, denote by $\mathrm{Rot}_\pi^\alpha(\boldsymbol a)$ the rotation of $\boldsymbol a$ around the axis $\pi^\perp$ by the angle $\alpha$.
\end{notation}

Let us define a~continuous map
\begin{gather*}
\phi:\mathrm{Gr}(1,\langle\hat{\boldsymbol e}_{n-1}\rangle^\perp)\times\mathrm{Gr}(1,\langle\hat{\boldsymbol e}_{n-1}\rangle^\perp)\times(T_pM^{n-1}\setminus\langle\hat{\boldsymbol e}_{n-1}\rangle^\perp)\to\mathrm{Gr}(n-2,T_pM^{n-1})\times\mathrm{Gr}(n-2,T_pM^{n-1}),\\
\phi(\ell_1,\ell_2,\boldsymbol a)=(\langle\ell_1^\perp\cap\langle\hat{\boldsymbol e}_{n-1}\rangle^\perp,\boldsymbol a\rangle,\langle\ell_2^\perp\cap\langle\hat{\boldsymbol e}_{n-1}\rangle^\perp,\mathrm{Ref}_{\langle\ell_1,\hat{\boldsymbol e}_{n-1}\rangle}(\boldsymbol a)\rangle)
\end{gather*}
(fig. \ref{fig:04}) and let $\ell\in\mathrm{Gr}(1,\langle\hat{\boldsymbol e}_{n-1}\rangle^\perp),\ \boldsymbol a\in T_pM^{n-1}\setminus\langle\hat{\boldsymbol e}_{n-1}\rangle^\perp$ be such that $\langle\ell^\perp\cap\langle\hat{\boldsymbol e}_{n-1}\rangle^\perp,\boldsymbol a\rangle\in V$. Then we have $\langle\ell^\perp\cap\langle\hat{\boldsymbol e}_{n-1}\rangle^\perp,\boldsymbol a\rangle=\langle\ell^\perp\cap\langle\hat{\boldsymbol e}_{n-1}\rangle^\perp,\mathrm{Ref}_{\langle\ell,\hat{\boldsymbol e}_{n-1}\rangle}(\boldsymbol a)\rangle$, so
$$\phi(\ell,\ell,\boldsymbol a)=(\langle\ell^\perp\cap\langle\hat{\boldsymbol e}_{n-1}\rangle^\perp,\boldsymbol a\rangle,\langle\ell^\perp\cap\langle\hat{\boldsymbol e}_{n-1}\rangle^\perp,\boldsymbol a\rangle)$$
is an element of $V\times V$. Since $V$ is open, the preimage $\phi^{-1}(V\times V)$ is an open neighborhood of $(\ell,\ell,\boldsymbol a)$. Thus it contains contains a~product of non-empty open sets $W_1\times W_2\times W_3$, where $W_1,W_2\subseteq\mathrm{Gr}(1,\langle\hat{\boldsymbol e}_{n-1}\rangle^\perp)$ are neighborhoods of $\ell$ and $W_3\subseteq(T_pM^{n-1}\setminus\langle\hat{\boldsymbol e}_{n-1}\rangle^\perp)$ is a~neighborhood of $\boldsymbol a$. Moreover, since $\phi(\ell_1,\ell_2,\boldsymbol a)=\phi(\ell_1,\ell_2,\lambda\boldsymbol a)$ for every $\lambda\neq 0$, we may assume that $W_3$ is the interior of a~generalized cone intersected with $U$.\\

Let $\ell_1\in W_1,\ \ell_2\in W_2,\ \boldsymbol a\in W_3$ and define $\boldsymbol a'\colonequals\mathrm{Ref}_{\langle\ell_1,\hat{\boldsymbol e}_{n-1}\rangle}(\boldsymbol a),\ \boldsymbol a''\colonequals\mathrm{Ref}_{\langle\ell_2,\hat{\boldsymbol e}_{n-1}\rangle}(\boldsymbol a')$ (fig. \ref{fig:04}). In light of the definition of $V$, it follows from \thref{lem:05} that $f\vert_{\langle\ell_1^\perp\cap\langle\hat{\boldsymbol e}_{n-1}\rangle^\perp,\boldsymbol a\rangle}$ is invariant under action of $\mathrm O(n-3,\mathbb R)$ with hyperplane of revolution $\ell_1^\perp\cap\langle\hat{\boldsymbol e}_{n-1}\rangle^\perp$. In particular, this group contains the reflection across the common fixed-point subspace, which can be viewed as a~restriction of $\mathrm{Ref}_{\langle\ell_1,\hat{\boldsymbol e}_{n-1}\rangle}$. Similarly, $f\vert_{\langle\ell_2^\perp\cap\langle\hat{\boldsymbol e}_{n-1}\rangle^\perp,\boldsymbol a'\rangle}$ is invariant under $\mathrm{Ref}_{\langle\ell_2,\hat{\boldsymbol e}_{n-1}\rangle}$, which implies
$$f(\boldsymbol a'')=f(\boldsymbol a')=f(\boldsymbol a).$$
Now, observe that
$$\boldsymbol a''=\mathrm{Ref}_{\langle\ell_2,\hat{\boldsymbol e}_{n-1}\rangle}(\mathrm{Ref}_{\langle\ell_1,\hat{\boldsymbol e}_{n-1}\rangle}(\boldsymbol a))=(\mathrm{Ref}_{\langle\ell_2,\hat{\boldsymbol e}_{n-1}\rangle}\circ\mathrm{Ref}_{\langle\ell_1,\hat{\boldsymbol e}_{n-1}\rangle})(\boldsymbol a)=\mathrm{Rot}_{\langle\ell_1,\ell_2\rangle}^{2\angle\ell_1\ell_2}(\boldsymbol a),$$
which eventually gives us
$$f\left(\mathrm{Rot}_{\langle\ell_1,\ell_2\rangle}^{2\angle\ell_1\ell_2}(\boldsymbol a)\right)=f(\boldsymbol a)$$
for every $\ell_1\in W_1,\ \ell_2\in W_2,\ \boldsymbol a\in W_3$. It means that the graph of $f$ (i.e. the surface $M^{n-1}$) is locally invariant on $W_3$ under action of $\mathrm O(n-2,\mathbb R)$ with common fixed-point subspace $\langle\hat{\boldsymbol e}_{n-1}\rangle^\perp$. Indeed, if we fix $\ell_1=\ell$ and let $\ell_2$ vary over $W_2$, we can rotate $\boldsymbol a$ in any direction by any sufficiently small angle. In particular, the series expansion of $f$ at $p$, as long as it is defined, is invariant under the aforementioned action of $\mathrm O(n-2,\mathbb R)$, which reads
\begin{equation}\label{eq:02}q_f(\boldsymbol x)=a\langle\boldsymbol x,\boldsymbol x\rangle+b{x_{n-1}}^2,\quad a,b\in\mathbb R\end{equation}
and thus
$$c_f(\boldsymbol x)=x_{n-1}(a\langle\boldsymbol x,\boldsymbol x\rangle+b{x_{n-1}}^2),\quad a,b\in\mathbb R.$$

\begin{figure}[h]
\includegraphics{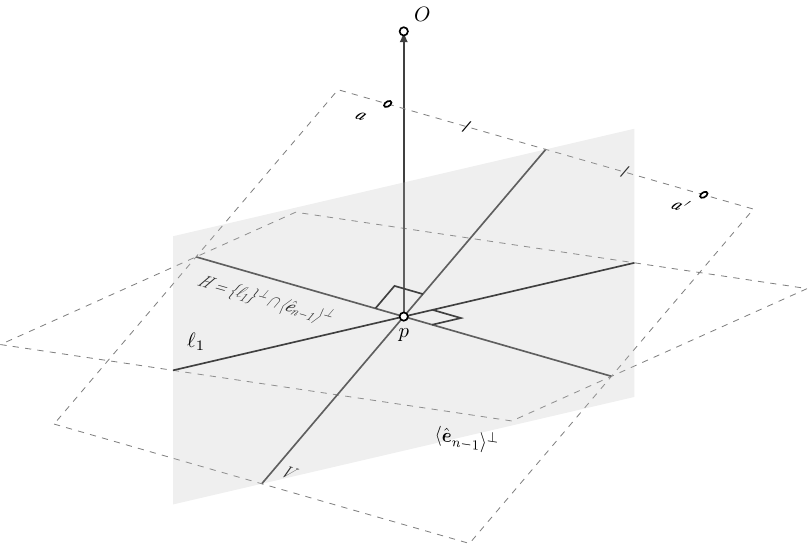}
\caption{The construction of $\boldsymbol a'$}
\label{fig:04}
\end{figure}

\begin{remark}
Our considerations so far show that at every point $p\in M^{n-1}$ with positive definite second fundamental form, the series expansion of $M^{n-1}$, as long as it is defined, admits a~symmetry group $\mathrm O(n-2,\mathbb R)$. Under the additional assumption that $M^{n-1}$ is locally strongly convex, such hypersurfaces have already been classified for $n=4$ (e.g. in \cite{Lu2005}). But since they may take a~complicated form of warped products, even such a~result gives no straightforward solution to our problem, not to mention higher dimensions, where to the author's best knowledge such a~classification is still an open problem.
\end{remark}

\section{Proof of the main theorem}

With this result at hand, we are ready to prove our main theorem:

\begin{proof}[Proof of \thref{thm:01}]
Denote by $V\subseteq\mathrm{Gr}(n-2,T_pM^{n-1})$ the set of hyperplanes $H$ such that $M^{n-1}\cap\aff(\{H,O\})$ admits an axis of affine revolution $\ell$ perpendicular to $H\cap\langle\hat{\boldsymbol e}_{n-1}\rangle^\perp$. Clearly $V$ is closed (cf. \cite[Lemma 2.7]{10.2140/gt.2021.25.2621}) and has non-empty interior (\thref{lem:08}).

\begin{lemma}\thlabel{lem:09}
In the above setting, we have either $V=\mathrm{Gr}(n-2,T_pM^{n-1})$ or $q_f\equiv 0$.
\end{lemma}

Again, we have to consider separately the special case $n=4$ and the generic case $n\geq 5$.

\begin{proof}[Proof of \thref{lem:09} for $n\geq 5$]
The projective quadric $Q_f^{n-2}\colonequals\{\boldsymbol x\in T_pM^{n-1}:q_f(\boldsymbol x)=0\}$ does not contain any linear subspace of dimension $n-3$ unless it is reducible \cite[Theorem 22.13]{harris1992algebraic}. From \eqref{eq:02} it can be readily seen that the latter implies $a=0$, which reads $c_f(\boldsymbol x)=b{x_{n-1}}^3$. Now, if $b=0$ then $q_f\equiv 0$, and we are done. Otherwise $c_f\vert_H$ vanishes precisely on $H\cap\langle\hat{\boldsymbol e}_{n-1}\rangle^\perp$ for every $H\in\mathrm{Gr}(n-2,T_pM^{n-1})$ and thus $V=\mathrm{Gr}(n-2,T_pM^{n-1})$.
\end{proof}

\begin{proof}[Proof of \thref{lem:09} for $n=4$]
Suppose that $V\neq\mathrm{Gr}(2,T_pM^3)$ and let $\pi\in\partial V$. Then there exists a~convergent sequence of planes $\pi_k\to\pi$ such that $M^3\cap\aff(\{\pi_k,O\})$ admits an axis of affine revolution $\ell_k$ perpendicular to some line in $\pi_k\cap Q_f^2$. After passing to a~subsequence, without loss of generality we may assume that the sequence $\ell_k$ is convergent to some $\ell_*$ perpendicular to some line in $\pi\cap Q_f^2$. From \eqref{eq:02} it can be readily seen that the latter is different from $\pi\cap\langle\hat{\boldsymbol e}_3\rangle^\perp$. Moreover, a~simple geometric continuity argument shows that $\ell_*$ is the axis of affine revolution of $M^3\cap\aff(\{\pi,O\})$ (cf. \cite[Lemma 2.7]{10.2140/gt.2021.25.2621}). Now, if $\ell_*\neq\ell$, then $M^3\cap\aff(\{\pi,O\})$ admits two different axes of affine revolution and hence is an ellipsoid (\thref{lem:02}). In particular, $c_f\vert_\pi\equiv 0$. On the other hand, if $\ell_*=\ell$, then $\ell$ is perpendicular to two different lines in the plane $\pi$, so it is perpendicular to the plane $\pi$ itself. Again, it implies $c_f\vert_\pi\equiv 0$ (\thref{lem:03}). Hence $c_f\vert_\pi\equiv 0$ for every $\pi\in\partial V$. However, $c_f\vert_\pi$ can vanish for at most $3$ different planes $\pi$ unless $q_f\equiv 0$ and the assertion follows.
\end{proof}

\begin{definition}[{\cite[II.3]{nomizu1994affine}}]\thlabel{def:01}
Let $f:M\to\mathbb R^{m+1}$ be a~non-degenerate hypersurface immersion. It is well known that there exists a~canonical choice of a~transversal vector field $\xi$ called the \emph{affine normal field} or \emph{Blaschke normal field} \cite[Definition~II.3.1]{nomizu1994affine}. The affine normal vector field $\xi$ gives rise to the induced connection $\nabla$, the affine fundamental form $h$, which is traditionally called the \emph{affine metric}, and the affine shape operator $S$ determined by the formulas
\begin{gather*}
D_XY=\nabla_XY+h(X,Y)\xi,\\
D_X\xi=-SX.
\end{gather*}
We shall call $(\nabla,h,S)$ the \emph{Blaschke structure} on the hypersurface $M$ \cite[Definition~II.3.2]{nomizu1994affine}. From Codazzi equation for $h$ we see that the cubic form
\begin{equation}\label{eq:05}C(X,Y,Z)\colonequals(\nabla_Xh)(Y,Z)\end{equation}
is symmetric in $X$, $Y$ and $Z$ \cite[II.4]{nomizu1994affine}.
\end{definition}

\begin{claim}\thlabel{lem:10}
It turns out that the condition $c_f\equiv 0$ implies that the cubic form $C$ also vanishes at $p$. It is by no means obvious, as \eqref{eq:05} can hardly be expressed in the extrinsic coordinate system (cf. \cite[1.4.3]{arxiv.2202.09894}). However, we can readily see that $C$ depends only on $J_p^3f$. Indeed, the affine normal field $\xi$ depends only on $J_p^3f$ (cf. \cite[Example~II.3.3]{nomizu1994affine}) and the affine metric $h$ depend only on $J_p^2f$ (cf. \cite[Example~II.3.3, Proposition~II.2.5]{nomizu1994affine}). Hence the covariant derivative
$$(\nabla_Xh)(Y,Z)\colonequals X(h(Y,Z))-h(\nabla_XY,Z)-h(Y,\nabla_XZ)$$
depends only on $J_p^3f$. In particular, if another function $g:T_pN^{n-1}\supset U\to\mathbb R$ satisfies $J_p^3f=J_p^3g$, then the cubic form of $M$ and the cubic form of $N$ coincide at $p$. Now, since $c_f\equiv 0$, the canonical parametrization of $M^{n-1}$ osculates up to the terms of \nth{3} order the parametrization of the unit sphere
$$g(\boldsymbol x)=1-\sqrt{1-\langle\boldsymbol x,\boldsymbol x\rangle},$$
for which the cubic form $C$ vanishes identically (cf. \cite[Corollary~II.4.2]{nomizu1994affine}). This concludes the argument.\hfill$\square$
\end{claim}

The following lemma may be considered a~counterpart of \thref{lem:02}:

\begin{lemma}\thlabel{lem:07}
A body of affine $2$-revolution $K\subset\mathbb R^m$, $m\geq 4$, admitting three different codimension $2$ hyperplanes of affine revolution, admits a~codimension $1$ hyperplane of affine revolution (i.e. is a~body of affine $1$-revolution).
\end{lemma}

\begin{proof}
Let $G$ be the affine symmetry group of $K$. Since by \cite[Lemma 2.2]{10.2140/gt.2021.25.2621} $G$ is affinely conjugated to a~subgroup of $\mathrm O(m,\mathbb R)$, without loss of generality we may assume that $G\subseteq\mathrm O(m,\mathbb R)$. In particular, each codimension $2$ hyperplane of affine revolution $H_i$ of $K$ gives rise to a~subgroup $G_i\subset G$ isomorphic to $\mathrm O(m-2,\mathbb R)$.\\

It turns out that the proof of \thref{lem:07} reduces to a~quite simple but tedious linear algebra problem. The key idea is the following: if the hyperplanes $H_i$ were pairwise transversal, then the orbit of a~generic point under action of $G$ would be of dimension $m-1$, which means that $\partial K$ would be a~sphere. Otherwise i.a. $G_2,G_3$ share a~common representation space $H_2+H_3$ of dimension $m-1$ and a~common fixed point subspace $H_2^\perp\cap H_3^\perp$ of dimension $1$, in which case the subgroup $\langle G_2,G_3\rangle\subseteq G$ generated by $G_2,G_3$ is by \thref{lem:02} isomorphic to $\mathrm O(m-1,\mathbb R)$.\\

Firstly we will show that $\dim H_2+H_3=m-1$, unless $\partial K$ is a~sphere. Let $p\in\partial K$ be any point on the boundary of $K$. Since $\partial K$ is invariant under $G$, we have $T_p(Gp)\subseteq T_p(\partial K)$, where $Gp$ is the orbit of $p$ under action of $G$. Now, if $\dim T_p(Gp)=m-1=\dim T_p(\partial K)$ for some $p\in\partial K$, then $\partial K$ is a~sphere and we are done. Hence we may assume that for every $p\in\partial K$ we have $\dim T_p(Gp)\leq m-2$. Observe that $T_p(G_ip)$ is a~codimension $3$ hyperplane parallel to $H_i\cap\langle p\rangle^\perp$, unless $H_i\subset\langle p\rangle^\perp$. Moreover $T_p(G_1p)+T_p(G_2p)+T_p(G_3p)\subseteq T_p(Gp)$. It follows that for every $p\in\partial K$ we have
\begin{equation}\label{eq:03}\dim H_1\cap\langle p\rangle^\perp+H_2\cap\langle p\rangle^\perp+H_3\cap\langle p\rangle^\perp\leq\dim T_p(Gp)\leq m-2.\end{equation}

Let $L$ be an arbitrary codimension $3$ hyperplane contained in $H_1$ but not in $H_2,H_3$. Then $L^\perp$ is a~subspace of dimension $3$ and $H_2^\perp\cap L^\perp,H_3^\perp\cap L^\perp$ are its subspaces of dimension at most $1$. Hence there exists a~plane $\pi$ contained in $L^\perp$ and transversal to $H_2^\perp,H_3^\perp$. Let $p,q\in\partial K$ be its basis. Observe that $H_1\cap\langle p\rangle^\perp=L=H_1\cap\langle q\rangle^\perp$, whereas $H_i\cap\langle p\rangle^\perp\neq H_i\cap\langle q\rangle^\perp$, $i=2,3$. Indeed, otherwise
$$3=\codim H_i\cap\langle p\rangle^\perp=\codim H_i\cap\langle p\rangle^\perp\cap\langle q\rangle^\perp=\codim H_i\cap\pi^\perp=\dim H_i^\perp+\pi=4,$$
a contradiction. Denote
$$H_p\colonequals L+H_2\cap\langle p\rangle^\perp+H_3\cap\langle p\rangle^\perp,\quad H_q\colonequals L+H_2\cap\langle q\rangle^\perp+H_3\cap\langle q\rangle^\perp.$$
Clearly, $H_i\cap\langle p\rangle^\perp\subsetneq H_i\cap\langle p\rangle^\perp+H_i\cap\langle q\rangle^\perp\subseteq H_i$, $i=2,3$, and since the dimension of the left-hand side and the right-hand side differs by one, the last inclusion must be in fact an equality. Thus
$$H_2+H_3=H_2\cap\langle p\rangle^\perp+H_2\cap\langle q\rangle^\perp+H_3\cap\langle p\rangle^\perp+H_3\cap\langle q\rangle^\perp\subseteq H_p+H_q.$$
Now, by \eqref{eq:03} we have $\dim H_p,\dim H_q\leq m-2$. Moreover, $\dim H_p\cap H_q\geq\dim L=m-3$, which implies
$$\dim H_p+H_q=\dim H_p+\dim H_q-\dim H_p\cap H_q\leq(m-2)+(m-2)-(m-3)=m-1.$$
Comparing the dimensions of the left-hand side and the right-hand side of $H_2\subsetneq H_2+H_3\subseteq H_p+H_q$ yields $\dim H_2+H_3=m-1$ and hence also $\dim H_2^\perp\cap H_3^\perp=\dim(H_2+H_3)^\perp=1$.\\

Finally, observe that $\mathbb R^m$ can be viewed as a~direct sum $(H_2+H_3)\oplus(H_2^\perp\cap H_3^\perp)$ of representation spaces of a~subgroup $\langle G_2,G_3\rangle\subseteq G$ generated by $G_2,G_3$. Indeed,
$$H_2+H_3=\bigcup_{v\in H_3}H_2+v=\bigcup_{v\in H_3}H_2+\mathrm{proj}_{H_2^\perp}(v)$$
is clearly invariant under $G_2$ and a~similar argument shows that it is also invariant under $G_3$. Moreover, both $G_2$ and $G_3$ act trivially on $H_2^\perp\cap H_3^\perp$. Now, we have
$$\mathrm{SO}(m-2,\mathbb R)\simeq(G_2)^0\vert_{H_2+H_3}\subsetneq\langle G_2,G_3\rangle^0\vert_{H_2+H_3}\subseteq\mathrm{SO}(m-1,\mathbb R),$$
and since $\mathrm{SO}(m-2,\mathbb R)$ is a~maximal connected subgroup of $\mathrm{SO}(m-1,\mathbb R)$ \cite[Lemma~4]{10.2307/1968975}, it follows that $\langle G_2,G_3\rangle^0\vert_{H_2+H_3}\simeq\mathrm{SO}(m-1,\mathbb R)$. Therefore $\langle G_2,G_3\rangle\simeq\mathrm O(m-1,\mathbb R)$, which concludes the proof.
\end{proof}

\begin{remark}
Note that \thref{lem:02} (without the superfluous symmetry assumption) reads: if the affine symmetry group $G$ of a~compact domain $K\subset\mathbb R^m$, $m\geq 3$, contains two different subgroups $H_1,H_2<G$ affinely conjugated to $\mathrm O(m-1,\mathbb R)$, it contains a~subgroup affinely conjugated to $\mathrm O(m,\mathbb R)$. Further, \thref{lem:07} reads: if the affine symmetry group $G$ of a~compact domain $K\subset\mathbb R^m$, $m\geq 4$, contains three different subgroups $H_1,H_2,H_3<G$ affinely conjugated to $\mathrm O(m-2,\mathbb R)$, it contains a~subgroup affinely conjugated to $\mathrm O(m-1,\mathbb R)$. But even if $G$ contains \emph{a priori} only two different subgroups $H_1,H_2<G$ affinely conjugated to $\mathrm O(m-2,\mathbb R)$, in general it should contain infinitely many of them, as the image of $H_1$ under any inner automorphism of $G$ arising from conjugation by $g\in H_2$ is again a~subgroup affinely conjugated to $\mathrm O(m-2,\mathbb R)$. Let us state then a~more general question:

\begin{question}
Does a~compact domain $K\subset\mathbb R^m$, $m\geq 3$, admitting $\lfloor m/(m-k)\rfloor+1$ different codimension $k$ hyperplanes of affine revolution, admit a~codimension $k-1$ hyperplane of affine revolution, $1<k<m$?
\end{question}

\noindent To the author's best knowledge, the answer is not known.
\end{remark}

Recall \thref{lem:09} which says that either $V=\mathrm{Gr}(n-2,T_pM^{n-1})$ or $q_f\equiv 0$. In the first case (i.e. $V=\mathrm{Gr}(n-2,T_pM^{n-1})$), we can repeat the geometric argument from \S\ref{sec:01} to show that actually the whole hypersurface $M^{n-1}$ is invariant under action of $\mathrm O(n-2,\mathbb R)$ with common fixed-point space $\aff(\{\langle\hat{\boldsymbol e}_{n-1}\rangle,O\})$. Hence by \thref{lem:07} all such points $p\in M^{n-1}$ lie on at most two different planes $\pi_1,\pi_2$, unless $M^{n-1}$ is a~body of affine revolution.\\

Finally, we pass to the second case (i.e. $q_f\equiv 0$). Let $p\in M^{n-1}$ be the point attaining the maximal Euclidean distance from the origin. It means that $M^{n-1}$ is contained in some sphere tangent to $M^{n-1}$ at $p$. In particular, the second fundamental form of $M^{n-1}$ at $p$ majorizes the second fundamental form of the sphere, and thus $M^{n-1}$ is strongly convex on some open neighborhood of $p$. Let $U\subseteq M^{n-1}$ be a~maximal open neighborhood of $p$ where the second fundamental form of $M^{n-1}$ is positive definite. We already know from \thref{lem:10} that the cubic form of $M^{n-1}$ vanishes identically on $U\setminus(\pi_1\cup\pi_2)$.

\begin{lemma}[{Maschke, Pick, Berwald \cite[Theorem~II.4.5]{nomizu1994affine}}]\thlabel{lem:06}
Let $f:M\to\mathbb R^{m+1}$, $m\geq 2$, be a~non-degenerate hypersurface with Blaschke structure. If the cubic form \eqref{eq:05} vanishes identically, then $f(M)$ is hyperquadric in $\mathbb R^{m+1}$.
\end{lemma}

It follows from \thref{lem:06} that $U$ is contained in some hyperquadric $Q^{n-1}$. Now, suppose that $\partial U$ is non-empty and let $p\in\partial U$. Since $Q^{n-1}$ is locally strongly convex, the second fundamental form of $Q^{n-1}$ at $p$ is positive definite. However, the second fundamental form of $M^{n-1}$ at $p$ is equal to the latter and thus it is also positive definite on some open neighborhood of $p$, which contradicts the definition of $U$. It follows that $U=M^{n-1}$, which concludes the proof.
\end{proof}

\begin{remark}\thlabel{rem:01}
In our proof, we used the additional assumption that $K$ is origin-symmetric only to know that all the axes of affine revolution pass through some fixed point, which implies some nice geometric structure of $M^{n-1}$, determined by its series expansion of order $3$. This significantly simplified our argument, which after all required no algebraic computations. Nevertheless, there are e.g. certain partial differential equations of order $5$, satisfied whenever $g$ is a~local parametrization of a~surface of affine revolution. When applied to $f\vert_\pi$ for every plane $\pi\in\mathrm{Gr}(2,T_pM^3)$, they would yield a~system of polynomial equations in partial derivatives of $f$. However, it is beyond the scope of human to obtain, not to mention to solve. Therefore any approach along those lines would badly need the assistance of a~supercomputer.
\end{remark}

\section*{Acknowledgments}

I would like to thank Prof. D.~Ryabogin for giving me this problem to consider and for many inspiring discussions, my doctoral advisor Prof. M.~Wojciechowski for his help in completing all the details and the exceptional amount of effort spent on editorial work, Prof. A.~Weber for helping me understand the algebraic topology behind the proof, Prof. J.~Wiśniewski for bringing the theorem of Bertini to my attention and Prof. T.~Mostowski for a~brief introduction to affine differential geometry.

\bibliography{references}{}
\bibliographystyle{amsplain}

\end{document}